\documentclass[reqno]{amsart}
\usepackage{amsmath}
\usepackage{amssymb}
\usepackage{amsthm}
\usepackage{bussproofs}
\usepackage{enumitem}
\usepackage[backend = biber, style = alphabetic]{biblatex}
\usepackage{comment}
\usepackage{stmaryrd}
\usepackage{mathtools}
\usepackage{tikz-cd}
\usepackage{cmll}
\usepackage{xcolor}
\usepackage{bbold}

\theoremstyle{plain}
\newtheorem{thm}{Theorem}
\newtheorem*{thm*}{Theorem}

\newtheorem{prop}[thm]{Proposition}
\newtheorem{cor}[thm]{Corollary}
\newtheorem*{cor*}{Corollary}

\theoremstyle{definition}
\newtheorem{defn}[thm]{Definition}

{}
\bibliography{references}

\definecolor{cg}{rgb}{0.0, 0.42, 0.24}

\newcommand{\Set}{\mathbf{Set}}

\DeclareMathOperator{\D}{D}
\DeclareMathOperator{\G}{G}

\newcommand{\cf}[1]{\mathbf{#1}}

\newsavebox{\pullback}
\sbox\pullback{%
	\begin{tikzpicture}%
	\draw (0,0) -- (1ex,0ex);%
	\draw (1ex,0ex) -- (1ex,1ex);%
	\end{tikzpicture}}

\DeclareMathOperator{\PR}{P}

\DeclareMathOperator{\Ob}{Ob}
\newcommand{\Cat}{\mathbf{Cat}}
\newcommand{\Pos}{\mathbf{Pos}}

\DeclareMathOperator{\Amp}{\&}

\tikzset{mid vert/.style={/utils/exec=\tikzset{every node/.append style={outer sep=0.8ex}},
		postaction=decorate,decoration={markings,
			mark=at position 0.5 with {\draw[-] (0,#1) -- (0,-#1);}}},
	mid vert/.default=0.75ex}

\begin{document}
\title{Dialectica Categories over Heyting Algebras}
\thanks{An extended abstract of this work appeared under the title \emph{Very Small Dialectica Categories}.}



\providecommand{\NoCaseChange}[1]{#1}
\author[C. Bloomfield, P. Jipsen, and V. de Paiva]{\NoCaseChange{%
\normalfont\upshape
\begin{minipage}[t]{0.33\textwidth}\centering
\textbf{Colin Bloomfield} \\
{\small Independent Researcher} \\
{\small\texttt{colinbloomfield1@gmail.com}}
\end{minipage}%
\begin{minipage}[t]{0.33\textwidth}\centering
\textbf{Peter Jipsen} \\
{\small Chapman University} \\
{\small\texttt{jipsen@chapman.edu}}
\end{minipage}%
\begin{minipage}[t]{0.33\textwidth}\centering
\textbf{Valeria de Paiva} \\
{\small Topos Institute} \\
{\small\texttt{valeria@topos.institute}}
\end{minipage}%
}}

\begin{abstract}
Categorification---the process of constructing a categorical model of a piece of mathematics---often identifies a common abstraction that connects formerly unrelated but known structures. In the case of de Paiva's categorification of G\"odel's Dialectica interpretation, 
we find that its specialization to partial orders produces  (functorial) embeddings of Heyting algebras into residuated lattices that
appear to have been overlooked. For the non-categorical audience, we present this specialization and take care to reproduce the original proofs in the algebraic setting. Along the way we obtain results particular to this algebraic setting: an embedding lacking an evident adjoint in de Paiva's general construction acquires a definable one here; a single Dialectica tensor validates contraction in the intuitionistic construction $\D$ yet refutes it in the classical variant $\G$; and, over $\mathrm{ZF}$, the poset reflection $P\D(\Set)$ collapses onto the four-element algebra $P\D(2)$ exactly when the Axiom of Choice holds.
\end{abstract}

\maketitle

\section{Introduction}

The Dialectica categories of Valeria de Paiva \cite{dePaiva1991} generalize the translation of Intuitionistic Heyting Arithmetic into G\"odel's System T. G\"odel's original intent~\cite{godel1990} was to provide a relative consistency proof of Peano arithmetic. One of the original motivations for de Paiva's categorification  was to extract from the interpretation its categorical structure, which then provided one of the first non-collapsed categorical models of linear logic.

Let $ \cf{C} $ be a category with finite limits. The \textbf{Dialectica Category} $ \D(\cf{C}) $ associated to $ \cf{C} $, has as objects  triples consisting of a pair of objects $ U, X \in \Ob(\cf{C}) $ and a monomorphism $\alpha \colon A \rightarrowtail U \times X $. We write such an object as $ (\alpha, U, X) $.

A morphism from $ (\alpha, U, X) $ to $ (\beta, V, Y) $ consists of a pair of morphisms of $ \cf{C} $, $ (f, F) $, $ f \colon U \to V $, $ F \colon U \times Y \to X $ such that a non-trivial condition is satisfied: pulling $ \alpha $ back along $ \langle \pi_{1},F \rangle $ and $ \beta $ along $ f \times Y $, the first subobject is smaller than the second as in the following diagram
\begin{equation} \label{diag: Dial Morphism}
\begin{tikzcd}[column sep =50]
  & A^{\prime} \arrow[dashed,bend right = 30]{dl}{} \arrow[tail]{d}{\langle \pi_{1}, R \rangle^{-1}(\alpha)} \arrow{r}{} \arrow[dr, phantom, "\usebox\pullback" , pos=.05, color=black] & A \arrow[tail]{d}{\alpha}  \\
B^{\prime} \arrow[dr, phantom, "\usebox\pullback" , pos=.05, color=black] \arrow[tail]{r}{(f \times Y)^{-1}(\beta)} \arrow{d}{}  & U \times Y \arrow{r}{\langle \pi_{1},F \rangle} \arrow{d}{f \times Y} & U \times X \\
B \arrow[tail]{r}{\beta} & V \times Y & 
\end{tikzcd}
\end{equation}

In the category $ \Set $, we can identify subobjects with subsets, and the pullback as the preimage operation. Then the condition becomes: $ (f,F) \colon \alpha \to \beta $ is a morphism of $\D(\Set)$ iff whenever $ (u,F(u,y)) \in \alpha $, then $ (f(u),y) \in \beta $. That is, 
$$
\langle \pi_{1},F \rangle^{-1}(\alpha) \leq (f \times Y)^{-1}(\beta).
$$
We pay special attention throughout to $ \D(\Set) $, both for its historical role and as the jumping-off point for the categorification.

 Composition of morphisms is defined in Dialectica categories, and it is shown that $\D(\cf{C}) $ is a category. 
 If $\cf{C}$ is Cartesian closed (with well-behaved coproducts), then $ \D(\cf{C})$ has sufficient structure, so that the associated logic satisfies all the rules of $ILL$
 . This is no easy task, especially since to capture the proof theory of $ILL$ in $\D(\cf{C})$, the additional operations on $\D(\cf{C})$ need to be natural.

Since Dialectica categories have strong connections to logic and not every logician is a categorical one, an original motivation for this note was to provide a more gentle introduction to Dialectica categories via small models. Just as it is helpful to learn about the finite cyclic groups, dihedral groups, the quaternions, etc.\ when first learning group theory, we ask: Are there compelling finite Dialectica categories that can improve our intuitions?  


The Dialectica construction $\D(\cf{C}) $ requires a category $\cf{C}$ with at least finite limits, and a category with products and two parallel morphisms necessarily has infinitely many objects and so, we consider partial-order categories, i.e.\ those where there is at most one morphism between objects, and isomorphic objects are identified. When $\cf{C}$ is a poset, the additional operations defined on $\D(\cf{C})$ remain compelling, and we no longer need to concern ourselves with the naturality of such operations, since every diagram trivially commutes in a poset category.

There is another reason for considering such categories: In (proof-theoretic) categorical logic, a category $\cf{C}$ provides semantics for a logic as follows: Each atomic formula $ A $ of the logic is assigned to an object of $\cf{C}$, and the connectives of the logic are interpreted as natural operations on the category which inductively assign an object $ [\phi] $ of $\cf{C}$ to each formula $\phi$. 
As a simple example, if your logic of interest has meets, then your category $\cf{C}$ providing semantics needs binary products, and $[\phi \wedge \psi] $ is interpreted as $[\phi] \times [\psi] $.

Then $\phi \vdash \psi $ is said to hold if and only if, for every such assignment of atomic formulas to objects of $\cf{C}$, there is a morphism from $[\phi] $ to $[\psi]$. The naturality of the operations interpreting the connectives ensures that when we take the poset reflection\footnote{The poset reflection is a functor $P \colon \Cat \to \Pos $, which maps each category $\cf{C}$ to a poset $P(\cf{C})$ obtained first by identifying parallel morphisms and then identifying isomorphic objects. One can see that functors $F \colon \cf{C} \to \cf{D}$ respect the resulting equivalence classes, and so $P(F)$ is the result of applying $F$ to representatives of the classes.} of $\cf{C}$, $ P(\cf{C})$, the associated ordered algebra provides semantics for the same logic. The difference is, morphisms of the category are intended to represent proofs, and the algebra captures derivability (or provability), i.e., only the existence of a proof.

Although we lose all proof-theoretic information captured by distinct parallel morphisms, ordered-algebraic semantics provides a powerful and greatly simplified view into the associated logic.
These small models turn out to be more than a teaching aid. Restricting to partial orders, the Dialectica construction becomes a functorial embedding of Heyting algebras into residuated lattices---a specialization that, to our knowledge, has gone unremarked---and the finite examples expose structure that the general theory hides. Three findings stand out: 1) de Paiva's embedding $\iota(u) = (u,u,u)$, which she found to lack an evident adjoint in general, turns out to have a perfectly definable one over a Heyting algebra; 2) a single Dialectica tensor validates contraction in $\D$ yet refutes it in its classical twin $\G$; and 3) most surprisingly, over $\mathrm{ZF}$ the identity $ P \D(2) = P\D(\Set)$ is equivalent to the Axiom of Choice, while its classical analogue $ P \G(2) = P\G(\Set)$ fails outright, even under choice.

\section{Dialectica Over Meet-Semilattices}

Recall that the Dialectica construction $\D(\cf{C}) $  over a category $\cf{C} $ (\cite{dial_cats1989}) requires that the category $\cf{C} $ has binary products and pullbacks. If $\cf{C} $ is a poset, then the categorical product $a \times b $ is necessarily the meet $ a \wedge b $. Moreover, if $ a \leq c $ (i.e.\ $\exists f \colon a \to c) $ and $ b \leq c $ (i.e.\ $\exists g: b \to c$), then the pullback  of $f,g$ is simply $a \wedge b $. Applying these observations to the Dialectica morphism shown in Diagram (\ref{diag: Dial Morphism}), we have:
\begin{defn}
 The Dialectica category over a meet-semilattice $\cf{M}=(M, \leq, \wedge, 1)$, written as $\D (\cf{M})$, is a preorder whose objects are triples $(a, u, x)$ where $a, u, x$ are elements of $\cf{M}$ and $a\leq u\wedge x$. In $\D (\cf{M})$, $(a, u, x) \leq (b, v, y)$ if and only if the following three inequalities hold:
\begin{itemize}
\item[(i)] $u \leq v$
\item[(ii)] $ u \wedge y \leq x$ {and} 
\item[(iii)] $a \wedge (u \wedge y) \leq b\wedge (u \wedge y)$
\end{itemize}   
\end{defn}


Since algebraists prefer to work with ordered, instead of pre-ordered algebras we define the \textit{Dialectica Algebra}, $ D^{\prime}(\cf{M}) $, to be the sub-order of $D(\cf{M}) $ defined as 
\begin{equation*}
    \D^{\prime}(\cf{M}) \coloneqq \{ (a,u,x) \in M : a \leq x \leq u \}. 
\end{equation*}
\begin{prop}\label{prop: PD iso Dprime}
    For $\cf{M} $ a meet-semilattice, $P \D (\cf{M}) $ is order-isomorphic to $ D^{\prime}(\cf{M})$. And so, as categories, $\D^{\prime}(\cf{M}) $ is equivalent to $\D(\cf{M}) $.
\end{prop}
\begin{proof}
    Every object of $\D(\cf{M})$ is equivalent to one in $\D^{\prime}(\cf{M})$, since $(a,u,x) \equiv (a,\, u,\, u\wedge x)$ and $a \leq u\wedge x \leq u$. The preorder restricted to $\D^{\prime}(\cf{M})$ is antisymmetric: if $(a,u,x) \leq (b,v,y)$ and conversely, with $a\leq x\leq u$ and $b\leq y\leq v$, then $u=v$, $x=y$, and $a=b$. Hence $\D^{\prime}(\cf{M})$ is a skeleton of $\D(\cf{M})$, order-isomorphic to $P\D(\cf{M})$.
\end{proof}

In the sequel, we will take care to define operations on $ \D(\cf{M})$ so that they restrict to operations on $ \D^{\prime}(\cf{M}) $, allowing us to work with $ \D^{\prime}(\cf{M}) $ instead of $P \D(\cf{M})$.


%

If $\bf{M}$ has a smallest element $ 0 $ then $\top \coloneqq (0,1,0)$ is the top of $\D (\cf{M}) $ and $\bot \coloneqq (0,0,0) $ is a bottom element. Moreover, let $ i \coloneqq (1,1,1) $ and for $ \alpha = (a,u,x) $, and $\beta = (b,v,y) $, define $\alpha \otimes \beta \coloneqq (a \wedge b, u \wedge v, x \wedge y)$. Observe that $ (\D(\cf{M}),\otimes, i) $ is an ordered commutative monoid or in categorical language, $(\otimes, i)$ a symmetric monoidal tensor product.

Moreover, since the tensor product $\otimes $ is defined on objects via coordinate-wise meets, it is immediate that the associated logic satisfies the following inference rules:
\begin{equation*}
	\AxiomC{$ A, \Gamma \vdash B $} 
        \RightLabel{dupl}
	\UnaryInfC{$ A,A, \Gamma \vdash B $}
	\DisplayProof
	\hskip 1.5 em
	\AxiomC{$ A,A, \Gamma \vdash B $}
        \RightLabel{cont}
	\UnaryInfC{$A, \Gamma \vdash B$}
	\DisplayProof
\end{equation*}

The first rule is a restricted form of weakening we call duplication, and the second rule is contraction. Just as important are the rules the logic doesn't satisfy. We will see that the two-element Boolean algebra $\mathbf{2} $, $ \D(\mathbf{2})$ shown in Figure \ref{fig: D(2)} fails to satisfy the following more general rule of weakening:
\begin{equation*}
	\AxiomC{$ A, \Gamma \vdash B $} 
        \RightLabel{weak}
	\UnaryInfC{$ A,C, \Gamma \vdash B $}
	\DisplayProof
\end{equation*}
And so, the logics associated to classes of $\D^{\prime}(\cf{M})$ algebras seem to be relevance logics: That is, antecedents cannot be exhausted, but they need to play a role in the derivation of the consequent.

We would now like a linear implication $\multimap$ in the Dialectica preorder. In the categorical semantics, what is needed is for the tensor $\otimes$ on $\D(\cf{C})$ to be closed, and in \cite{dePaiva1991}, it is shown this holds whenever $\cf{C} $ is Cartesian closed. The order theoretic version of a Cartesian closed category is a residuated meet-semilattice $\cf{R}$. That is, $\cf{R} $ is a meet-semilattice such that for all $ x,y,z \in R $ the following condition is satisfied:
\begin{equation*}
    x \wedge y \leq z \quad \iff \quad y \leq x \rightarrow z.
\end{equation*}

\begin{prop}\label{prop: internal hom}
    The operation $\multimap:(\D \cf{R})^{op}\times \D \cf{R} \to \D \cf{R}$, defined for $\beta = (b, v, y) $ and $\gamma = (c, w, z) $ by
    \begin{equation*}
        \beta \multimap \gamma \coloneqq ((b \rightarrow c) \wedge s \wedge t, s, s \wedge t)
    \end{equation*}
    where $s = (v \rightarrow w) \wedge ((v \wedge z) \rightarrow y)$ and $ t = v \wedge z $ is the internal-hom in $\D (\cf{R})$. I.e.,  each $ \alpha \otimes {(\cdot)}$ in $\D(\cf{R}) $ is residuated with residual $\alpha \multimap {(\cdot)} $.
\end{prop}
\begin{proof}
    By Proposition~\ref{prop: PD iso Dprime} the internal-hom is determined only up to equivalence, and $\beta \multimap \gamma \equiv (a', s, t)$ with $a' = (b\rightarrow c)\wedge s\wedge t$; we verify the adjunction for this representative. As $\D(\cf{R})$ is a preorder, it suffices to show $\alpha \otimes \beta \leq \gamma \iff \alpha \leq (a', s, t)$ for every $\alpha = (a,u,x)$. The left-hand side unfolds to
    \begin{equation*}
        u\wedge v \leq w, \qquad u\wedge v\wedge z \leq x\wedge y, \qquad a\wedge b\wedge u\wedge v\wedge z \leq c,
    \end{equation*}
    and the right-hand side to $u \leq s$, $\ u\wedge t \leq x$, and $a\wedge u\wedge t \leq (b\rightarrow c)\wedge s$. Residuation makes these equivalent: $u\wedge v\leq w$ together with $u\wedge v\wedge z\leq y$ say exactly $u \leq (v\rightarrow w)\wedge((v\wedge z)\rightarrow y) = s$; $u\wedge v\wedge z\leq x\wedge y$ gives $u\wedge t\leq x$; and $a\wedge b\wedge u\wedge v\wedge z\leq c$ reads $a\wedge u\wedge t\leq b\rightarrow c$, which with $u\leq s$ yields the third condition.
\end{proof}

Consider $ \cf{R} = \mathbf{2} $. For simplicity, we write $ aux $ for $ (a,u,x) $. The preorder of $ \D(\mathbf{2}) $ is shown in Figure \ref{fig: D(2)}:
\begin{figure}
$$
\begin{tikzcd}[row sep = 15]
\top= 010 &  \\
i= 111 \ar{u} &   \\
a=011 \ar{u} &  \\
\bot= 001 \ar{u} \ar["\sim"]{r} & 000 \ar{ul} \ar{l}
\end{tikzcd}
$$ 
\caption{The residuated meet-semilattice $ P \D (\mathbf{2})$.}\label{fig: D(2)}
\end{figure}
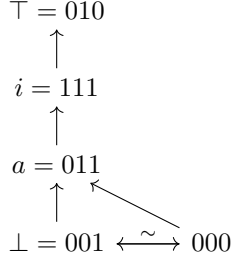
 If we identify isomorphic objects and let $ \bot = 000$, $a = 011$, $ i = 111$ and $\top = 010$, we have the following operation table for $P\D(2) = \D^{\prime}(2)$:
\begin{equation} \label{eqn: min noncom mult table}
\begin{array}{c | c c c c}
\otimes     & \bot & a & i & \top \\
\hline  \bot & \bot & \bot & \bot & \bot \\
a & \bot & a & a & \top \\
i & \bot & a & i & \top \\
\top & \bot & \top & \top & \top
\end{array}
\hskip 3 em
\begin{array}{c | c c c c}
\multimap    & \bot & a & i & \top \\
\hline       \bot & \top & \top & \top & \top \\
a & \bot & i & i & \top \\
i & \bot & a & i & \top \\
\top & \bot & \bot & \bot & \top
\end{array}
\end{equation}
The result is a four-element commutative, idempotent, involutive, residuated lattice, which is not integral, known as the Sugihara algebra $S_4$. \cite{sugihara1955}

Being a four-element chain, $\D^{\prime}(\mathbf{2})$ is in particular a Heyting algebra, and one may wonder whether every $\D^{\prime}(\cf{H})$ is. Already the three-element chain refutes this: $\D^{\prime}(\cf{3})$ is a ten-element, non-distributive lattice, hence not a Heyting algebra, as we confirmed by direct computation.

Moreover, assuming choice, $\D^{\prime}(\mathbf{2}) $ has the following striking connection to $\D (\Set) $, the motivating example for the original categorification:
\begin{thm}\label{PD(SET) = PD(2)}
    Over $\mathrm{ZF} $, $\D^{\prime}(2) = P \D(\Set) $ is equivalent to the Axiom of Choice.
\end{thm}
\begin{proof}
    By Proposition~\ref{prop: PD iso Dprime}, $ \D^{\prime}(2) = \PR\D(2) $, so it suffices to show that $ \PR\D(\Set) = \PR\D(2) $ if and only if $\Set$ satisfies choice. Let $1 = \{ \emptyset \} $, and overload $1$ to also denote $ \{ ( \emptyset, \emptyset) \} $ when specifying the Dialectica object $ (\alpha, U, X) = (1,1,1) $.

    First observe that the Dialectica objects in $ \D(\Set) $ of the form $ \{ (\alpha,U,X) : U,X \in \{0,1\} \} $ form a sub-Dialectica category isomorphic to $ \D(2) $: the categorical operations on $0$ and $1$ (exponentials, pullbacks, products, and coproducts) as sets correspond exactly to the logical operations on them as members of the two-element Boolean algebra.

    Given this, it suffices to show that every element of $ \D(\Set) $ is equivalent in $ \PR\D(\Set) $ to one of the objects in this sub-Dialectica category. Each $ (\alpha,U,X) $ of $ \D(\Set) $ falls into one of four classes:
    \begin{enumerate}
        \item the class of $ (0,1,0) $: all $ (\alpha,U,X) $ with $ U \neq \emptyset $ and $ X = \emptyset $;
        \item the class of $ (1,1,1) $: all $ (\alpha,U,X) $ with $ U \neq \emptyset $, $ X \neq \emptyset $, and $ \exists u \in U $ such that $ (u,x) \in \alpha $ for all $ x \in X $;
        \item the class of $ (0,1,1) $: all $ (\alpha, U, X) $ with $ U \neq \emptyset $, $ X \neq \emptyset $, and for all $ u \in U $ there is $ x \in X $ with $ (u,x) \notin \alpha $;
        \item the class of $ (0,0,0) $: all $ (\alpha, U, X) $ with $ U = \emptyset $.
    \end{enumerate}
    For classes 1, 2, and 4 it is a straightforward exercise to show, over $ \mathrm{ZF} $ alone, that an arbitrary element of each class lies in the same element of $ \PR\D(\Set) $ as its specified representative. For $ (\alpha, U, X) $ in class~(3) there is a morphism $ (f,F) \colon (0,1,1) \to (\alpha, U, X) $ over $ \mathrm{ZF} $: since $ U \neq \emptyset $ there is $ f \colon 1 \to U $, there is always a function $ F \colon 1 \times X \to 1 $, and since the relation on $ (0,1,1) $ is trivial the subobject condition holds. It remains to show that, for all $ (\alpha, U, X) $ of class~(3), a morphism $ (f, F) \colon (\alpha, U, X) \to (0,1,1) $ exists if and only if choice holds.

    Suppose the Axiom of Choice holds, and let $ (\alpha, U, X) $ be an arbitrary object of class~3. By choice there is a function $ F \colon U \times 1 \to X $ with $ (u, F(u,0)) \notin \alpha $ for all $ u \in U $. Then $ (f, F) \colon (\alpha, U, X) \to (0,1,1) $.

    Conversely, suppose $ \PR\D(\Set) \cong \PR\D(2) $, so that each element of class~3 represents the same object of $ \PR\D(\Set) $ as $ (0,1,1) $. Let $ \{ X_u \}_{u \in U} $ be an arbitrary family of non-empty sets, put $ X = \bigcup_{u \in U} X_u $ and $ \alpha = \{ (u, x) : u \in U,\ x \in X,\ x \notin X_u \} $. Since each $ X_u $ is non-empty, for every $ u \in U $ there is $ x \in X $ with $ (u, x) \notin \alpha $, so $ (\alpha, U, X) $ is of class~(3); hence there is a Dialectica morphism $ (f, F) \colon (\alpha, U, X) \to (0,1,1) $. For $ (f,F) $ to satisfy the subobject condition, $ F(\,\cdot\,, 0) \colon U \to X $ must be a choice function.
\end{proof}
\begin{cor}
    Assuming choice, the logic associated to $\D(\Set) $ is the logic associated to $S_4$. As a consequence, the logic $\vdash_{\D(\Set)}$ is strictly stronger than $ILL$.
\end{cor}

The comparison with $ILL$ here rests on de Paiva's theorem that every Dialectica category is a model of intuitionistic linear logic \cite{dePaiva1991}: $\D(\Set)$ therefore validates at least $ILL$, and its identification with $S_4$ shows it validates strictly more. Having drawn this consequence, we resume the order-theoretic development, reconstructing the remaining connectives---the additive conjunction and the ``of course'' modality---directly on $\D(\cf{H})$, and verifying by hand the structure the general theory supplies categorically.

When $\cf{R} = \cf{H} $ is a Heyting algebra, $\D(\cf{H}) $ moreover has finite meets: the top $\top = (0,1,0)$ is the empty meet, and the binary meet is given as follows.
\begin{prop}\label{prop: meet}
    For $\cf{H}$ a Heyting algebra, $\alpha = (a,u,x)$, $\beta = (b,v,y)$, and $w = u\wedge v$, the binary meet in $\D(\cf{H})$ is
    \begin{equation*}
        \alpha \Amp \beta = \bigl(((w \wedge x) \rightarrow a ) \wedge ((w \wedge y)\rightarrow b) \wedge w \wedge (x \vee y),\ w,\ w \wedge (x \vee y)\bigr).
    \end{equation*}
\end{prop}
\begin{proof}
    First $\alpha\Amp\beta \leq \alpha$: we have $w\leq u$, $\ w\wedge x \leq w\wedge(x\vee y)$, and the first-coordinate condition holds because $((w\wedge x)\rightarrow a)\wedge(w\wedge x)\leq a$. By symmetry $\alpha\Amp\beta\leq\beta$.

    Now suppose $\gamma = (r,s,t)$ satisfies $\gamma\leq\alpha$ and $\gamma\leq\beta$; unwinding, $s\leq u\wedge v$, $\ s\wedge x\leq t$, $\ s\wedge y\leq t$, and $r\wedge s\wedge x\leq a$, $\ r\wedge s\wedge y\leq b$. Then $\gamma\leq\alpha\Amp\beta$: the first condition is $s\leq w$; the second, $s\wedge w\wedge(x\vee y)\leq t$, follows since $s\wedge(x\vee y) = (s\wedge x)\vee(s\wedge y)\leq t$ by distributivity; and for the third it suffices, by symmetry, to show $r\wedge s\wedge w\wedge(x\vee y)\leq (w\wedge x)\rightarrow a$, i.e.\ $r\wedge s\wedge w\wedge x\leq a$, which is immediate from $r\wedge s\wedge x\leq a$.
\end{proof}

Now we would like to search for a suitable operation $ {!} \colon {\D(\cf{H}) \to \D(\cf{H})} $ to interpret the ``of course'' modality of intuitionistic linear logic. 
Since such an operation must be an interior operator---the algebraic specialization of a comonad---a natural first place to look is among monoidal embeddings $ \iota \colon \cf{H} \to \D(\cf{H}) $ together with a right adjoint $ f $, from which one forms $ {!} \coloneqq \iota \circ f $. Two such embeddings arise as specializations to $ \D(\cf{H}) $ of observations made for $ \D(\cf{C}) $ with $ \cf{C} $ Cartesian closed in Section~1.4 of \cite{dePaiva1991}.

The first is $ \iota(a) = (0,a,1) $, which carries meets to the tensor, since $ \iota(a \wedge b) = (0, a \wedge b, 1) = \iota(a) \otimes \iota(b) $, and admits the right adjoint $ f(a,u,x) = u $, as $ (0,u,1) \leq (b,v,y) $ iff $ u \leq v $. The induced operation $ !(a,u,x) \coloneqq \iota \circ f(a,u,x) = (0,u,1) $ is an interior operator, and it even satisfies weakening. However, the embedding $ \iota $ fails to preserve the unit, $ \iota(1) = (0,1,1) \neq i $, so $ {!}\,i = (0,1,1) \neq i $, and thus promotion fails: from $ \vdash i $ one cannot derive $ \vdash {!}\,i $.

The second is $ \iota(u) = (u,u,u) $, which is strong monoidal, carrying meets to the tensor and top to the unit ($ \iota(1) = (1,1,1) = i $). Of this embedding de Paiva remarks that it ``does not seem to have an easy adjoint'' \cite[\S1.4]{dePaiva1991}. In the order-theoretic specialization, however, residuation supplies one.
\begin{prop}\label{prop: G adjoint}
    Let $ \cf{H} $ be a Heyting algebra. The embedding $ \iota \colon \cf{H} \to \D(\cf{H}) $, $ \iota(w) = (w,w,w) $, has a right adjoint $ f \colon \D(\cf{H}) \to \cf{H} $ given by $ f(a,u,x) = u \wedge (x \rightarrow a) $.
\end{prop}
\begin{proof}
    For $ w \in \cf{H} $, the inequality $ (w,w,w) \leq (a,u,x) $ holds iff $ w \leq u $ and $ w \wedge x \leq a $; by residuation the latter is $ w \leq (x \rightarrow a) $. Hence $ (w,w,w) \leq (a,u,x) $ iff $ w \leq u \wedge (x \rightarrow a) = f(a,u,x) $.
\end{proof}
The induced comonad $ {!}_{\iota} \coloneqq \iota \circ f $ is then a modality on $ \D(\cf{H}) $,\footnote{The third coordinate of $ (t,t,t) $ is determined only up to the preorder of $ \D(\cf{H}) $: one has $ (t,t,t) \equiv (t,t,x') $ for every $ x' \geq t $, and we take the normal-form representative with $ a \leq x \leq u $.}
\begin{equation*}
    {!}_{\iota}(a,u,x) = (t,t,t), \qquad t = u \wedge (x \rightarrow a),
\end{equation*}
and, $ \iota $ being strong monoidal, the adjunction $ \iota \dashv f $ is a linear/non-linear model (its right adjoint $ f $ is then automatically lax monoidal), so $ {!}_{\iota} = \iota \circ f $ is a linear exponential comonad---a \emph{linear category} in Benton's terminology \cite[Cor.~8]{benton1994}. By contrast, composing $ \iota $ with the projection rather than its adjoint---the naive choice $ (a,u,x) \mapsto (u,u,u) $---is not even an interior operator, failing dereliction whenever $ a < u \wedge x $, as with $ (u,u,u) = i \nleq (0,1,1) $ at $ (0,1,1) \in \D(2) $. Thus the specialization furnishes, for free, an adjoint the general construction may lack.

The modality we take as our primary object of study arises not from an adjunction but directly from the additive structure of $ \D(\cf{H}) $. Informally, in linear logic, $!A$ means that we can ``use'' the argument $A$ as many times as we want, including none. If we were allowed infinitary formulas, we could express this as
\begin{equation*}
    i \Amp A \Amp (A \otimes A) \Amp (A \otimes A \otimes A) \Amp \ldots
\end{equation*}
However, in $ \D(\cf{H}) $, we have contraction, and so the above expression is reducible to $ i \Amp A $. Using the definitions of $i$ and $\Amp$ above we define
\begin{equation*}
    !(a,u,x) \coloneqq i \Amp (a,u,x) = (((u \wedge x) \rightarrow a) \wedge u, u, u) = ((x \rightarrow a) \wedge u,u,u).
\end{equation*}
Like $ {!}_{\iota} $, and unlike the first embedding, this $ ! $ preserves the unit, $ {!}\,i = (((1 \wedge 1) \rightarrow 1) \wedge 1, 1, 1) = i $, so promotion is unobstructed. It is nonetheless distinct from $ {!}_{\iota} $: for instance $ !(0,1,1) = (0,1,1) $ whereas $ {!}_{\iota}(0,1,1) = (0,0,0) $ in $ \D(2) $.
\begin{prop}\label{prop: bang}
    The $!$ operation defines an interior operator (comonad) on $\D(\cf{H}) $ satisfying $ !(a,u,x) \otimes {!(b,v,y)} \leq {!((a,u,x)} \otimes (b,v,y)) $.
\end{prop}
\begin{proof}
    Write $ \alpha = (a,u,x) $ and $ \beta = (b,v,y) $; recall $ !(a,u,x) = ((x \rightarrow a) \wedge u, u, u) $ and that $ \cf{H} $ is a Heyting algebra.

    \emph{Dereliction, $ {!\alpha} \leq \alpha $.} The conditions $ u \leq u $ and $ u \wedge x \leq u $ are immediate, and for the third, $ (x \rightarrow a) \wedge u \wedge x \leq a \wedge u \wedge x $ since $ x \wedge (x \rightarrow a) \leq a $.

    \emph{Monotonicity.} Suppose $ \alpha \leq \beta $, so $ u \leq v $ and $ u \wedge y \leq x $. For $ {!\alpha} \leq {!\beta} $, i.e.\ $ ((x \rightarrow a) \wedge u, u, u) \leq ((y \rightarrow b) \wedge v, v, v) $, the conditions $ u \leq v $ and $ u \wedge v \leq u $ hold, and the third reduces to $ y \wedge u \wedge v \wedge (x \rightarrow a) \leq b $. Since $ u \wedge y \leq x $ and $ t \mapsto (t \rightarrow a) $ is antitone, $ u \wedge y \wedge (x \rightarrow a) \leq u \wedge y \wedge ((u \wedge y) \rightarrow a) \leq a $, hence $ u \wedge y \wedge (x \rightarrow a) \leq a \wedge (u \wedge y) \leq b \wedge (u \wedge y) \leq b $, using the third inequality of $ \alpha \leq \beta $. As $ y \wedge u \wedge v \wedge (x \rightarrow a) \leq u \wedge y \wedge (x \rightarrow a) $, the claim follows.

    \emph{Idempotence.} Put $ a' = (x \rightarrow a) \wedge u $, so $ {!\alpha} = (a', u, u) $. Then $ {!!\alpha} = ((u \rightarrow a') \wedge u, u, u) = (a', u, u) = {!\alpha} $, since $ u \wedge (u \rightarrow a') = u \wedge a' = a' $ (as $ a' \leq u $). Together with dereliction and monotonicity, this shows $ ! $ is an interior operator.

    \emph{Compatibility with $ \otimes $.} Computing both sides coordinatewise,
    \begin{align*}
        {!\alpha} \otimes {!\beta}
          &= \bigl( u \wedge v \wedge (x \rightarrow a) \wedge (y \rightarrow b),\ u \wedge v,\ u \wedge v \bigr) \\
          &\leq \bigl( u \wedge v \wedge ((x \wedge y) \rightarrow (a \wedge b)),\ u \wedge v,\ u \wedge v \bigr)
           = {!(\alpha \otimes \beta)},
    \end{align*}
    where the middle inequality holds because, in any Heyting algebra, $ (x \rightarrow a) \wedge (y \rightarrow b) \leq (x \wedge y) \rightarrow (a \wedge b) $.
\end{proof}

All together we have the following result.
\begin{thm}\label{thm: bang ILL}
    If $\cf{H} $ is a Heyting algebra, then the associated logic $\vDash_{\D(\cf{H})}$ satisfies the rules of Intuitionistic Linear logic with the ``of course'' modality.
\end{thm}
\begin{proof}
    The residuated tensor $(\otimes, i)$ (Proposition~\ref{prop: internal hom}) and the meet $\Amp$ (Proposition~\ref{prop: meet}) interpret the multiplicative and additive connectives, so $\D(\cf{H})$ models $ILL$. For the modality: dereliction $!\alpha \leq \alpha$ and $!\alpha \otimes {!\beta} \leq {!(\alpha\otimes\beta)}$ are Proposition~\ref{prop: bang}; contraction $!\alpha \leq {!\alpha}\otimes{!\alpha}$ is immediate since $\otimes$ is idempotent; and weakening holds because $!(a,u,x) = ((x\rightarrow a)\wedge u, u, u) \leq i$. Promotion is immediate from Proposition~\ref{prop: bang}: for a context $\Gamma = \phi_1,\dots,\phi_n$, write $\Delta = {!\phi_1}\otimes\cdots\otimes{!\phi_n}$; idempotence and the monoidal law give $\Delta = {!!\phi_1}\otimes\cdots\otimes{!!\phi_n} \leq {!\Delta}$, so $\Delta \leq \alpha$ yields $\Delta \leq {!\alpha}$ by monotonicity.
\end{proof}

\section{The $\G$ Construction}

De Paiva's companion construction $\G(\cf{C})$ \cite{dePaiva1991} provides semantics for \emph{classical} linear logic. It agrees with $\D(\cf{C})$ on objects, and its morphisms differ only in that the backward map $F \colon U \times Y \to X$ is replaced by a map $Y \to X$ with no dependence on $U$. For a meet-semilattice $\cf{C}$ this strengthens the second clause of the order: $\G(\cf{C})$ has the same objects $(a,u,x)$ with $a \leq u \wedge x$ as $\D(\cf{C})$, but $(a,u,x) \leq (b,v,y)$ now requires
\begin{itemize}
    \item[(i)] $u \leq v$,
    \item[(ii$'$)] $y \leq x$, and
    \item[(iii)] $a \wedge (u \wedge y) \leq b \wedge (u \wedge y)$.
\end{itemize}
As (ii$'$) implies clause (ii) of $\D(\cf{C})$, every relation of $\G(\cf{C})$ holds in $\D(\cf{C})$; and unlike $\D(\cf{C})$, the preorder $\G(\cf{C})$ is already a partial order.

Over $\cf{C} = \mathbf{2}$ the arrow from $000$ to $111$ present in $\D(2)$ is lost, and $000$ is no longer identified with $001$. The result is the pentagon $N_5$ of Figure~\ref{fig: G(2)}, in which $000$ is incomparable to $111$.
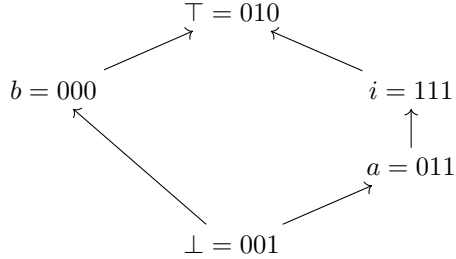
\begin{figure}[h]
$$
\begin{tikzcd}[row sep = 15]
 & \top = 010 &  \\
b = 000 \ar{ur} &  & i = 111 \ar{ul}   \\
 &  & a = 011 \ar{u} \\
 &  \bot = 001 \ar{uul} \ar{ur} &
\end{tikzcd}
$$
\caption{The pentagon $\G(2) \cong N_5$.}\label{fig: G(2)}
\end{figure}

The multiplicative structure of $\G$ specializes de Paiva's Dialectica tensor \cite{dePaiva1991}, whose middle component is the exponential $X^V \times Y^U$. Writing the exponential with the Heyting residual, $X^V = (V \rightarrow X)$, this is
\begin{equation*}
    (a,u,x) \otimes (b,v,y) = \bigl( a \wedge b,\ u \wedge v,\ (v \rightarrow x) \wedge (u \rightarrow y) \bigr),
\end{equation*}
with unit $i = (1,1,1)$. This is the very tensor underlying $\D$: the coarser order of $\D$ collapses the residual third coordinate to the coordinatewise meet $x \wedge y$ used earlier, whereas the stricter order of $\G$ does not. Specializing to $\cf{C} = \mathbf{2}$, the tensor and internal-hom on the five objects of $\G(2)$ are
\begin{equation*}
\begin{array}{c | c c c c c}
\otimes & \bot & b & a & \top & i \\
\hline
\bot & \bot & \bot & \bot & \bot & \bot \\
b & \bot & \bot & b & b & b \\
a & \bot & b & a & \top & a \\
\top & \bot & b & \top & \top & \top \\
i & \bot & b & a & \top & i
\end{array}
\hskip 3 em
\begin{array}{c | c c c c c}
\multimap & \bot & b & a & \top & i \\
\hline
\bot & \top & \top & \top & \top & \top \\
b & b & \top & b & \top & \top \\
a & \bot & b & i & \top & i \\
\top & \bot & b & \bot & \top & \bot \\
i & \bot & b & a & \top & i
\end{array}
\end{equation*}
Unlike $\D(2)$, whose tensor is coordinatewise meet and hence idempotent, $\G(2)$ fails contraction: $b \otimes b = \bot$, while $b \nleq \bot$ since $\bot$ is the least element. Thus a single Dialectica tensor validates contraction in $\D$ but refutes it in $\G$, the difference owing entirely to the morphism condition.

Whereas $\PR\D(\Set) = \PR\D(2)$ under choice (Theorem~\ref{PD(SET) = PD(2)}), the analogous identification fails outright for $\G$, even over $\mathrm{ZFC}$.
\begin{thm}\label{thm: G Set neq G2}
    Over $\mathrm{ZF}$ alone, $\PR\G(\Set)$ is not equivalent to $\PR\G(2)$.
\end{thm}
\begin{proof}
    Two poset categories are equivalent if and only if they are isomorphic, so it suffices to exhibit an infinite descending chain in $\PR\G(\Set)$, since $\PR\G(2)$ has only five elements. For $n \in \mathbb{N}$ let $(\Delta n, n, n)$ be the object of $\G(\Set)$ with $n = \{0, \ldots, n-1\}$ and $\Delta n = \{(m,m) : m \in n\}$.

    Fix $n < m$. There is no morphism $(f,g) \colon (\Delta n, n, n) \to (\Delta m, m, m)$: for arbitrary $f \colon n \to m$ and $g \colon m \to n$, since $n < m$ there are $a \neq b$ in $m$ with $g(a) = g(b)$, whence $(g(a), g(a)), (g(a), g(b)) \in \Delta n$, yet at least one of $(f(g(a)), a)$ and $(f(g(a)), b)$ lies outside $\Delta m$, violating the subobject condition. Conversely there is always a morphism $(f,g) \colon (\Delta m, m, m) \to (\Delta n, n, n)$: take $f(a) = \min\{a, n-1\}$ and $g = \mathrm{id}_n$; if $(a, g(b)) = (a,b) \in \Delta m$ then $a = b \leq n-1$, so $(f(a), b) = (a,b) \in \Delta n$. Hence $[(\Delta m, m, m)] < [(\Delta n, n, n)]$ in $\PR\G(\Set)$ whenever $n < m$, giving the descending chain of Figure~\ref{fig: G chain}, which begins at $i = [(\Delta 1,1,1)]$ and is bounded below by $[(0,1,1)]$.
\end{proof}
\begin{figure}[h]
$$
\begin{tikzcd}[row sep = scriptsize]
 & {[0,1,0]} &  \\
{[0,0,0]} \ar{ur} &  & i = [\Delta 1,1,1] \ar{ul}   \\
 &  & {[\Delta 2,2,2]} \ar{u} \\
 &  & {[\Delta 3,3,3]} \ar{u} \\
 &  & \vdots \ar{u} \\
 &  & {[0,1,1]} \ar{u} \\
 &  {[0,0,1]} \ar{uuuuul} \ar{ur} &
\end{tikzcd}
$$
\caption{The infinite descending chain in $\PR\G(\Set)$.}\label{fig: G chain}
\end{figure}
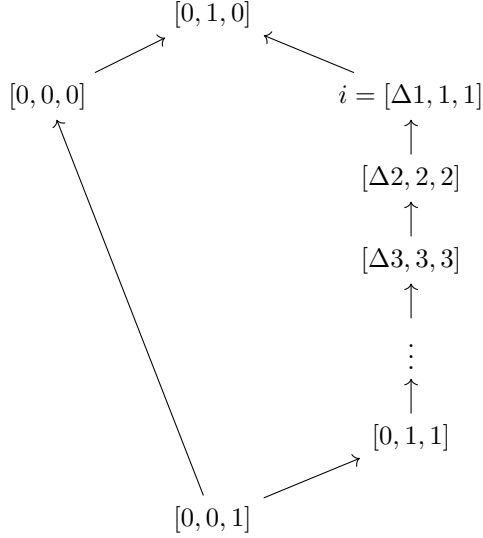

The failure of $\PR\G(\Set) = \PR\G(2)$ and the shape of the $\G$-tensor are two faces of one fact. The intuitionistic condition~(ii) of $\D$ is permissive and identifies a great deal---collapsing $\D(\Set)$ onto the four elements of $\D^{\prime}(2)$ under choice, and the residual third coordinate of the tensor onto a plain meet, so that $\D(2)$ satisfies contraction. The classical condition~(ii$'$) of $\G$ identifies far less: the objects $(\Delta n, n, n)$ that $\D$ would run together stay distinct---whence the descending chain---and the residual likewise survives in the tensor, so that contraction fails. The two anomalies of $\G$ are thus one: its morphisms identify too little to collapse them.

\section{Conclusion and Future Work}

We have specialized the Dialectica construction to partial orders and reproduced its basic theory algebraically, obtaining functorial embeddings of Heyting algebras into residuated lattices. Along the way we obtained several results specific to the specialization: de Paiva's second embedding $\iota(u) = (u,u,u)$ acquires a definable right adjoint (Proposition~\ref{prop: G adjoint}), giving a linear exponential comonad; a single Dialectica tensor validates contraction in $\D$ but refutes it in $\G$, the difference owing to the morphism condition alone; and, over $\mathrm{ZF}$, $\D^{\prime}(2) = P\D(\Set)$ is equivalent to the Axiom of Choice (Theorem~\ref{PD(SET) = PD(2)}). Several questions remain, however.

In investigating the algebraic case, we have converged on Heyting algebras, as suitable structures over which the Dialectica construction should occur. However, in \cite{dePaiva1991} one considers not bi-Cartesian closed categories, but categories $\cf{C} $ which are finitely complete, locally Cartesian closed, and with stable and disjoint coproducts. There is unfinished work to reconcile the different assumption sets: E.g.\ Does the poset reflection on these categories with weaker structure still result in Heyting algebras? Moreover, we have barely scratched the surface of understanding the algebras $ \D^{\prime}(\mathbf{H}) $ (and especially $\G(\cf{H})$), and it will be interesting to see if new results for the algebraic specialization---such as the residual/adjoint $f(a,u,x) = u \wedge (x \rightarrow a) $ to $ \iota(u) = (u,u,u)$---lift to the general Dialectica construction. 

In algebraic logic, one often considers a logic associated to a class of algebras. For example, the variety of Heyting algebras $\mathcal{H} $ provides semantics for intuitionistic logic in the following sense: We say $\phi \vDash_{\mathcal{H}} \psi$ if and only if, for each Heyting algebra $\cf{H} \in \mathcal{H} $, $\phi \vDash_{\cf{H}} \psi$. By definition, $\vDash_{\mathcal{H}} $ is the infimum of the logics defined by Heyting algebras, and it is the logic $\vDash_{\cf{H}}$ where $\cf{H} $ is the free Heyting algebra over the propositional atoms. Let $\mathcal{C}$ be the class of all (small) finitely complete locally Cartesian closed categories with stable and disjoint products and $\D(\mathcal{C}) $ all associated Dialectica Categories. From \cite{dePaiva1991}, ${\vDash_{\D (\mathcal{C})}} $ satisfies all the rules of $ILL$ but is $\vDash_{\D(\mathcal{C})}$ equal to $ILL$? Moreover, is there a $\cf{C} \in \mathcal{C}$, such that ${\vDash_{\D(\cf{C})}} = {\vDash_{\D(\mathcal{C})}}$? 

Let $\mathcal{C} $ be the class of all bi-Cartesian closed categories. Is ${\vDash_{\D (\mathcal{C})}} = {\vDash_{\D (\mathcal{H})}} $? Note that a $\cf{C} \in \mathcal{C} $ such that $\vDash_{\D(\cf{C})}$ does not satisfy contraction would disprove this claim. While such $\cf{C} $ are thought to exist, we do not currently have a witness. A model of $ZF$ that doesn't satisfy choice may seem like a good candidate but the function
\begin{equation*}
    F(u,x_{1},x_{2}) =
\begin{cases}
	x_{1} & \text{if }(u,x_{1}) \notin \alpha \\
	x_{2} & \text{otherwise}.
\end{cases}
\end{equation*}
defined for a Dialectica object $\alpha \subseteq U \times X $ is first-order definable over the signature $\{\in \} $ and so can be defined in $ZF$ via the Axiom Schema of Specification. And $F$ along with the Diagonal function $ f \colon U \to U \times U$ determines a morphism from $\alpha$ to $\alpha \otimes \alpha $.

From Theorem \ref{PD(SET) = PD(2)}, if $\Set$ is a model of $ZFC$, then $P\D P(\Set) = P \D(\Set) $, and one may wonder if over $\mathcal{C} $, $P \D P (\cdot) = P \D (\cdot) $? From Theorem \ref{PD(SET) = PD(2)} again this cannot hold in general since there are models of $ZF$ that do not satisfy choice. It would be interesting to characterize the subclass of $\mathcal{C} $ over which this result holds. Generalizations/categorifications of Theorem \ref{PD(SET) = PD(2)} and the above mentioned function $F$ may help narrow the search space for categories $\mathbf{C}$ where $\vdash_{\D(\mathbf{C})}$ does not satisfy contraction.

\printbibliography

\end{document}